\documentclass[12pt]{article}
\usepackage{amssymb}

\usepackage{amsmath}
\usepackage{amsthm}
\usepackage{tikz}
\usepackage{amscd}
\usepackage{latexsym}
\usepackage{mathrsfs}
\usepackage{color}
\usepackage{graphicx}
\usepackage[noadjust]{cite}
\usepackage[all,pdf]{xy}
\usepackage{enumerate}
\usepackage{doi}
\usepackage{enumitem}

\newtheorem{thm}{Theorem}[section]

\newtheorem{prop}[thm]{Proposition}
\newtheorem{lem}[thm]{Lemma}

\newtheorem*{rem}{Remark}

\DeclareMathOperator{\id}{Id}

\numberwithin{equation}{section}

\newcommand{\set}[1]{\left\{#1\right\}}

\begin{document}

\title{Global determinism of completely regular semigroups \footnote{The paper is supported by National Natural Science
Foundation of China (11971383, 11571278). }}
\author{{ \bf Baomin Yu$^{1, 2}$\footnote{The first author is supported by the Science and Technique Foundation of Weinan city (2020ZDYF-JCYJ-214) and  the Doctoral Research Foundation of Weinan Normal University (20RC16). E-mail: bmyuu@hotmail.com}}
   {\bf \quad Xianzhong Zhao$^{1}$\footnote{Corresponding author. E-mail: zhaoxz@nwu.edu.cn} }\\
    {\small $^1$ School of Mathematics and Data Science}\\
 {\small Shaanxi University of Science and Technology}\\
 {\small Xian, Shaanxi, 710021, P.R. China}\\
     {\small $^2$ School of Mathematics and Statistics},
   {\small Weinan Normal University} \\
   {\small Weinan, Shaanxi, 714099, P.R. China}}
\date{}
\maketitle
\begin{center}
\begin{minipage}{140mm}
\noindent\textbf{ABSTRACT.}
The power semigroup of a semigroup $ S $ is  the semigroup  of all nonempty subsets of $ S $ equipped with the naturally defined multiplication. A class $\mathcal{K} $ of semigroups is globally determined if any two members of $ \mathcal{K} $ with isomorphic globals are themselves isomorphic. The global determinability for various classes of semigroups has attracted some attention during the past 50 years. In this paper we prove  that the class of all completely regular semigroups is globally determined. This is an extension and generalization of a series of related results obtained by some other mathematicians.

\end{minipage}
\end{center}

\section{Introduction}
The power semigroup, or global, $P(S)$ of a semigroup $S$ is the semigroup of all nonempty subsets of $S$ equipped with the operation given by $ AB = \{ ab \mid a\in A, b\in B\}$.
A class $\mathcal{K}$ of semigroups is said to be \emph{globally determined} if any two members of $\mathcal{K}$ having isomorphic globals must themselves be isomorphic. The study of globals of semigroups was initiated by Tamura  and Shafer \cite{tamurashafer1967}. In 1967, Tamura \cite{tamura1967unsolved} asked whether the class of all semigroups is globally determined. A negative answer to Tamura's question was provided by Mogiljanskaja  \cite{mogiljanskaja1973}. Later on,
many authors devoted  to studying the structure of globals of semigroups as well as the global determinability for various classes of semigroups, see \cite{almeida,zhaogan2014,zhaogr2015,zhaogs2015,gould1984d,gould1984b,gould1984a,gould1984c,hamilton,kobayashi,mccarthyhayes,mogiljanskaja1973,
putcha1977,
putcha1979,tamura1967unsolved,tamura1984,tamura1984a,tamura1986,tamura1987,tamura1990,tamurashafer1967,yz2019,yzg,zgy}.

The following classes of  completely regular semigroups are
globally determined: semilattices \cite{gould1984a, kobayashi}, groups \cite{mccarthyhayes,tamurashafer1967}, Clifford semigroups \cite{zhaogan2014}, completely 0-simple semigroups \cite{tamura1986}, rectangular groups \cite{tamura1984a},
completely regular periodic monoid with irreducible identity \cite{gould1984c}, semigroups having regular globals \cite{zhaogr2015}, bands \cite{zhaogs2015,yzg} and normal orthogroups \cite{zgy}. As a generalization of these results,
we will prove in this paper that the class of all completely regular semigroups is globally determined.

Notice that \(S\) is naturally embedded into \(P(S)\) by the mapping $s \mapsto \{s\}$. Thus, we will identify $S$ with the singleton sets of $P(S)$ without further comments during the paper. We say that a class  $\mathcal{K}$  of semigroups satisfies the \emph{strong isomorphism property} if for any $S$, $T \in \mathcal{K}$ and each isomorphism $\psi \colon P(S) \to P(T)$, the restriction $\psi|_{S}$  of $\psi$ to $S$ is a bijection from $S$ onto $T$. Evidently, a class of semigroups is globally determined if it satisfies the strong isomorphism property.
The following classes of semigroups satisfy the strong isomorphism property: the classes of semilattices \cite{kobayashi}, completely regular periodic monoids with irreducible identity \cite{gould1984c}, Clifford semigroups \cite{zhaogan2014}, all semigroups having regular but non-idempotent globals \cite{zhaogr2015}, while the class of all bands  does not \cite{zhaogs2015}.

Recall that a semigroup $S$ is \emph{completely regular} if for every $a \in S$, there exists  $x \in S$ such that $a = axa$ and $ax=xa$.  If $S$ is a completely regular semigroup, then every $\mathscr{H}$-class of $S$ is a group, and so, for any $a\in S$, we denote by $a^{-1}$ the inverse of $a$ in $H_a$ and by $a^0$ the identity of $H_a$.  A completely regular simple semigroup  is  a \emph{completely simple semigroup}. A completely regular semigroup $S$ is completely simple if and only if $S$ satisfies the identity $a=(ax)^0a$. A semigroup $S$ satisfying the identity $a=ax$ (respectively, $a=xa$) is a \emph{left zero} (respectively,  \emph{right zero}) semigroup.

 As is well-known, if $S$ is a completely regular semigroup, then Green's relation $\mathscr{D=J}$ is a semilattice congruence on $ S $, that is, $S/\mathscr{D}$ is a semilattice, and ${D}_a $  is a completely simple semigroup for every $a\in S$. In the following, we denote  $ S $  by $ S = [Y ; S_\alpha, \alpha \in Y] $, or shortly, $ S = [Y; S_\alpha] $, where $Y\cong S/{\mathscr{D}}$ is  the \emph{structure semilattice} of $S$ and the  \emph{completely simple components} $S_\alpha$ of $S$ are the $\mathscr{D}$-classes of $S$ (see \cite{petrichreilly, howie1996book}).

For a semigroup $ S $, $ E(S)$ denotes the set of idempotents of $ S $, $ {L}_a(S) $  (respectively,  $ R_a(S) $, $ H_a(S) $, $ D_a(S) $) the $ \mathscr{L} $-class (respectively, $ \mathscr{R} $-class, $ \mathscr{H} $-class, $ \mathscr{D} $-class) containing the element $ a $ in $ S $. If $ \rho $ is a congruence on $ S $,  the $\rho$-class containing an element $ a $ in $ S $, namely the set $ \{x \in S \mid x\,\rho\,a\} $, is denoted by $ a \rho $. For any $ A\in {P}(S) $, we write $ A_\alpha = A \cap S_\alpha $ and $ \id A = \set{ \alpha \in Y \mid A_\alpha \neq \emptyset} $. Then it is easy to verify that
\begin{equation}\label{eqn:id}
\id (AB) = (\id A) (\id B) \quad (A, B \in {P}(S)).
\end{equation}

Also, we use the notations  $ \mathcal{CR} $ and $ \mathcal{CS}_0 $ for the class of all completely regular semigroups and the class of all completely simple semigroups which is neither a left nor a right zero semigroup, respectively.

By  \cite[Proposition 4.2]{zhaogs2015} and the proof of \cite[Theorem 6.8]{tamura1986},  we have
\begin{prop}\label{prop:cs0}
The class $\mathcal{CS}_0$ satisfies the strong isomorphism property.
\end{prop}

A {\emph {breakable semigroup}} is a semigroup $S$ for which  each nonempty subset of $S$ is a subsemigroup of it \cite{redei}, which plays a key role in the study of global determinability of classes of  completely regular semigroups.
A description of the structure of breakable semigroups is given by

\begin{lem}[{\cite[Theorem 50]{redei}}]\label{lem:a2}
A semigroup $S$ is breakable if and only if $S = \bigcup_{\alpha \in \Gamma} S_\alpha$, where $ \Gamma $ is a chain, the semigroups $ S_\alpha (\alpha \in \Gamma)$ are pairwise disjoint left or right zero semigroups, and if $ a \in S_\alpha $ and $ b \in S_\beta $ with $\alpha < \beta $ in $ \Gamma $ then $ ab = ba = a $.
\end{lem}

The above lemma tells us that $S$ is breakable if and only if $S$ satisfies
\[
(\forall a_1, a_2 \in S) \quad a_1 a_2  \in \{ a_1, a_2\}. \leqno{(A_2):}
\]
Pelik\'an  \cite{pelikan} introduced and studied semigroups satisfying
\[
(\forall a_1, a_2, \dots, a_n \in S) \quad a_1 a_2 \cdots a_n \in \{ a_1, a_2, \dots, a_n\}, \leqno{(A_n):}
\]
and showed that a semigroup satisfying condition $(A_n )$ for some even number $n$  is breakable  and that a semigroup satisfying condition $(A_n )$ for some odd number $ n $ greater than $1$ satisfies $( A_3 )$.
\begin{lem}[{\cite[Theorem 2 and Remark]{pelikan}}]\label{lem:a3}
A semigroup $S$ satisfies $(A_3)$ if and only if either $S$ is breakable or $S$ has the following form: $S = \bigcup_{\alpha \in \Gamma} S_\alpha$, where $ \Gamma $ is a chain with a maximal element $ \omega \in \Gamma $ such that the $ S_\alpha $ are pairwise disjoint subsemigroups of $S$, $S_\alpha$ is either a left or a right zero semigroup for $\alpha \in \Gamma$, $\alpha \neq \omega$ and $ S_\omega $ is a cyclic group of order two, moreover for any $ a \in S_\alpha, b \in S_\beta$ with $\alpha < \beta$, $ab = ba =a$.
\end{lem}

In the rest of the paper, we denote by $\mathcal{A}_2(S)$ and $\mathcal{A}_3(S)$ the sets of subsemigroups of $S$ satisfying $(A_2)$ and $(A_3)$, respectively. Also, we write
$  \overline{\mathcal{A}_2}(S)  = \{ A \in \mathcal{A}_2(S) \mid | \id A  | = 1 \}$.
Clearly, $\mathcal{A}_2(S) \subseteq \mathcal{A}_3(S) \subseteq EP(S) $, where $ EP(S) $ is the set of idempotents of $ P(S) $.

Throughout this paper, unless specified otherwise, we always assume that $S=[Y; S_\alpha]$ and $S'=[Y'; S'_{\alpha'}]$ are completely regular semigroups,  $ \psi $ is an isomorphism from $ {P}(S) $ onto $ {P}(S')$. We will show that $\psi(a) \in S'$ if $D_a$ is not a left or right zero semigroup. However, we can no longer guarantee that $\psi(a) \in S'$ in the other cases. Instead, we will construct a new mapping $\eta \colon S \to S'$ that coincides with $\psi$ for the $\mathscr{D}$-classes which are not left or right zero semigroups and finally, prove that $\eta$ is an isomorphism.

The rest of this paper is organized as follows. We start out with some characterizations of subsets $\mathcal{A}_2(S)$ and $\mathcal{A}_3(S)$ of $P(S)$ in Section 2, and  prove that the restrictions $\psi |_{\mathcal{A}_3(S)}$ and $ \psi |_{\mathcal{A}_2(S)}$ of $\psi$ to $\mathcal{A}_3(S)$ and $\mathcal{A}_2(S)$ respectively are bijections from $\mathcal{A}_3(S)$  onto $\mathcal{A}_3(S')$ and $\mathcal{A}_2(S)$ onto $\mathcal{A}_2(S')$,  respectively.
Then we show in Section 3 that there exists a semilattice isomorphism $\theta \colon Y \to Y'$ such that the restriction $\psi|_{P(S_\alpha)}$ of $\psi$ to $P(S_\alpha)$ is an isomorphism from $P(S_\alpha)$ onto $P(S'_{\theta(\alpha)})$ for all $\alpha \in Y$. In Section 4, we show that $\psi(a) \in S'$ if $D_a \in \mathcal{CS}_0$ or $a$ is not maximal for the natural partial order in $S$. Finally,  in Section 5 we construct a new mapping $\eta \colon S \to S'$ that close related to $ \psi $ and prove that it is an isomorphism. This shows that the class of all completely regular semigroups is globally determined.

For other notations and terminology not given in this article, the reader is referred to the books \cite{howie1996book, petrichreilly}. Also, the generalized continuum hypothesis is assumed in this paper.

\section{On the subsets $\mathcal{A}_2(S)$ and $\mathcal{A}_3(S)$ of $S$}

The following result is a direct consequence of Lemma \ref{lem:a3}.
\begin{lem}\label{lem:a30}
Let $ A \in \mathcal{A}_3(S) $. Then:
\begin{enumerate}
 \item $\ a^0 \in A $ for every $a \in A$;
 \item if $ B \subseteq A $ is such that $ B^2 = A $, then $ B = A $.
\end{enumerate}
\end{lem}

\begin{lem}\label{lem:a31}
Let $A \in \mathcal{A}_3(S)$ and $B \in P(S)$. Then:
\begin{enumerate}
 \item if $B^2 = A$ then $\id B = \id A$;
 \item if $\id B \subseteq \id A$ and $BA=AB=A$, then $B \subseteq A$;
 \item if $BA = B^2 = A$ then $B=A$.
% \item If $AS = BS$ and $BA = AB = A$, then  $B \subseteq A$.
\end{enumerate}
\end{lem}
\begin{proof}
(i) Suppose $B^2 = A$ and let $\alpha \in \id B$. By \eqref{eqn:id}, we have
\[
  \alpha = \alpha^2 \in (\id B)^2 = \id B^2 = \id A.
\]
Thus $  \id B \subseteq \id A $. Since  $\id A$ is a chain, and so does $\id B$. Therefore,
\[
	\id B = (\id B)^2 = \id B^2 = \id A.
\]

(ii) Suppose that $\id B \subseteq \id A$ and $BA= AB = A$. Let $a \in A_\alpha$ and $b \in B_\alpha$ where $\alpha \in \id B$. Then $ba \in B_\alpha A_\alpha \subseteq A_\alpha$. Since $A \in \mathcal{A}_3(S)$, by Lemma \ref{lem:a30}(i) we have  $(ba)^0 \in A_\alpha$. Since $A_{\alpha}$ is completely simple, $b = (ba)^0b \in A_\alpha B_\alpha \subseteq A_\alpha$.  Thus $B \subseteq A$.

(iii) If $BA = B^2 = A$, then $AB = B^3= BA = A$ and by part (i) we have $\id B = \id A$. It follows by part (ii) that  $B \subseteq A$. Then, by Lemma \ref{lem:a30}(ii) we have $B = A$.
\end{proof}

As a generalization of \cite[Lemma 2.5]{zgy}, we have
\begin{lem}\label{lem:a32}
Let $ A \in EP(S) $ be such that
\[
	(\forall B \in P(S)) \quad BA = B^2 = A  \Longrightarrow  B = A .
\]
Then the following statements are true:
  \begin{enumerate}
    \item $a^3 = a $ and $ a^0 \in A $ for all $a \in A$;
    \item $ ab \in \{ a, b, a^0, b^0 \} $ for all $a, b \in A$;
    \item $ \id A$ is a subchain of $Y$;
    \item if $ a^2 \neq a $ for some $ a \in A $, then $ H_a(S) \cap A = \{ a, a^0 \}$;
    \item $ A_\alpha $ contained either in a single $ \mathscr{L} $-class  or in a single $ \mathscr{R} $-class in $ S $ for any $ \alpha \in \id A$;
    \item if $ \alpha $ is not a maximal element in $ \id A $, then $ A_\alpha $ is either a left or a right zero semigroup. Moreover, if $ a \in A_\alpha$ and $ b \in A_\beta $ with $ \alpha < \beta $, then $ ab = ba =a $;
    \item if $\id A$ has a maximal element $ \omega $ and there is an element $ a \in A_\omega $ such that $ a^2 \neq a $, then $ A_\omega = \{ a, a^0 \} $ is a cyclic group of order two.
  \end{enumerate}
\end{lem}
\begin{proof}
Parts (i), (ii) and (iii) are direct consequences of  \cite[Lemma 2.5]{zgy} and its proof.

 (iv) Suppose that there exists $ a \in A $ such that $ a^2 \neq a $ and let $ b \in H_a(S) \cap A $. Then $ b^0 = a^0 $. According to part (i), $ b^2 = a^0 = a^2 \in A $. Then by part (ii), $ ab \in \{a, b, a^0 \} $.
 If $ ab = a $ then
 \[
 	b =a^0 b = a^2b = a(ab) = a^2 = a^0 .
\]
 If $ ab = b $ then
 \[
 	a = aa^0 =ab^2 =(ab)b = b^2  = a^2 ,
\]
and this contradicts the assumption that $ a^2 \neq a $. Thus $ ab \neq b $.
If $ ab = a^0 $ then
\[
	b = a^0b  = (ab)b = ab^2 = aa^0 =a.
\]
Hence, $ H_a(S) \cap A = \{ a, a^0 \} $.

 (v)  If there is  $ \alpha \in \id A$ such that $ A_\alpha $ contained neither in a single $ \mathscr{L} $-class nor a single $ \mathscr{R} $-class in $ S $, then there exist $a, b \in A_\alpha $ such that $ H_a(S), H_b(S)$ and $ R_a(S) \cap L_b(S) $ are pairwise disjoint. Since $ab \in R_a(S) \cap L_b(S)$, $ab \not\in H_a(S) \cup H_b(S)$. However, by part (ii) we have $ ab \in \{a, a^0, b, b^0\} \subseteq H_a(S) \cup H_b(S) $, a contradiction. Therefore, $ A_\alpha $ contained either in a single $ \mathscr{L} $-class  or a single $ \mathscr{R} $-class in $ S $ for all $\alpha \in \id A$.

 (vi) Suppose that $ \alpha, \beta \in \id A $ with $ \alpha<\beta$. To show that $ A_\alpha $ is either a left or a right zero semigroup, by part (v) we need only prove that $ a^2 = a $ for all $ a \in A_\alpha $. In fact, if there exists $ a \in A_\alpha $ such that $ a \neq a^2=a^0 $, let us consider the subset $ B = A \setminus \{ a^2\} $ of $ A $. Clearly $B^2 \subseteq BA\subseteq A$. Now let $b \in A$. If $b\not\in B$, then $b=a^2\in B^2$. If $b\in B$, we may assume $b \in B_\gamma$ with $\gamma\leq \alpha$, since $B_\gamma=A_\gamma$ for all $\gamma > \alpha$.  Take $c \in B_\beta$. Then by part (ii) $b^0=b^0c$ and so $b=bb^0=bb^0c=bc \in B^2$. It follows that $A\subseteq B^2$. Thus $B^2=BA=A$. By hypothesis we have $ B = A $, a contradiction.  Therefore $ a^2 = a $ for all $ a \in A_\alpha $ and so, $ A_ \alpha $ is either a left or a right zero semigroup.

 If $a \in A_\alpha$ and $ b \in A_\beta $, then by part (ii), both $ ab $ and $ ba $ are in $ \{ a, b, b^0 \} \cap A_\alpha = \{ a \} $. Consequently, $ ab = ba = a $.

 (vii) Assume that $ \omega $ is a maximal element of $ \id A $ and that $ a \in A_\omega $ is such that $ a^2 \neq a $. By part (v),  $ A_\omega \subseteq L_a(S) $ or $ A_\omega \subseteq R_a(S) $. Without loss of generality, suppose $ A_\omega \subseteq R_a(S) $. If there exists $ b \in A_\omega \setminus \{ a, a^0 \}$, then by part (iv) $ b \not \in H_a(S) $. Thus $ a^0 \neq b^0 $. Since $ a \, \mathscr{R} \, b $, $ a^0b^0 = b^0 $. Also, by part (ii), we have $ ab^0 \in \{ a, a^0, b^0 \} \cap H_b=\{b^0\}$. Hence $ab^0=b^0=a^0b^0$. Since the mapping $ x \mapsto xb^0 $ from $ H_a(S) $ onto $ H_b(S) $ is a bijection,  $ a=a^0=a^2 $, a contradiction. Therefore, $ A_\omega = \{ a, a^0 \}$ is a group of order two.
\end{proof}

By Lemmas \ref{lem:a31} and \ref{lem:a32}, we obtain the following characterization of $\mathcal{A}_3(S)$.
\begin{prop}\label{prop:a34}
Let $A \in EP(S)$. Then $A \in \mathcal{A}_3(S)$ if and only if
\[
	(\forall B \in P(S)) \quad B^2 = BA = A  \Longrightarrow B = A.
\]
\end{prop}

\begin{thm}\label{thm:a35}
  The restriction $ \psi|_{\mathcal{A}_3(S)} $ of $ \psi $ to $ \mathcal{A}_3(S) $ is a bijection from $ \mathcal{A}_3(S) $ onto $ \mathcal{A}_3(S') $.
\end{thm}
\begin{proof}
  Let $ A \in \mathcal{A}_3(S) $. Then $ A \in EP(S) $ and so $ \psi(A) \in EP(S') $. Suppose that $ X \in P(S') $ is such that $ X^2 = X \psi(A) = \psi(A) $. Then
%  $ \psi^{-1}(X) \in P(S) $ satisfies
$    (\psi^{-1}(X))^2 = \psi^{-1}(X) A = A $.
  By Proposition \ref{prop:a34}, $ \psi^{-1}(X) = A $ and so $ X = \psi(A) $. Hence, again by  Proposition \ref{prop:a34}, $ \psi(A) \in \mathcal{A}_3(S') $.

%\[
%	  (\forall X \in P(S')) \ X^2 = X \psi(A) = \psi(A) \Longrightarrow X = \psi(A).
%\]
%By  Proposition \ref{prop:a34} again, we have $ \psi(A) \in \mathcal{A}_3(S') $.

The dual argument shows that $ \psi^{-1}(X) \in \mathcal{A}_3(S) $ if  $ X \in \mathcal{A}_3(S') $.
\end{proof}

From the proof of \cite[Lemma 2.12]{zgy}, we have
\begin{lem}\label{lem:green1}
 Let $A, B \in P(S)$. Then
 \begin{enumerate}
\item if $A\, \mathscr{R}\, B$, then $AS = BS$;
\item if $A\,\mathscr{D}\, B$, then $\id A=\id B$.
\end{enumerate}
\end{lem}

\begin{lem}\label{lem:green2}
 Let $ A, B \in P(S) $ be such that
 \begin{enumerate}
     \item for any $a \in A $ there is $ b \in B $ such that  $a \,\mathscr{R} \, b$ in $ S $ and
     \item for any $ b \in B$ there is $ a \in A $ such that $ a \, \mathscr{R} \, b $ in $ S $,
 \end{enumerate}
  then $AS=BS$ and $ \psi(A) S' = \psi(B) S' $.
\end{lem}
\begin{proof}
  Let $ a \in A $. Then  $ a \, \mathscr{R} \, b $ in $ S $ for some $ b \in B $. Thus $\{a\} \, \mathscr{R} \, \{b\}$ in $ P(S) $, and so $ \psi(a) \, \mathscr{R} \, \psi(b)$ in $P(S') $. By Lemma \ref{lem:green1},  we have $aS=bS \subseteq BS$ and  $ \psi(a) S' = \psi(b) S'$. Then $AS = \bigcup_{a \in A} a S \subseteq B S$. Dually, we have $BS \subseteq AS$. Hence, $AS=BS$.

Also, since $ \psi(a) S' = \psi(b) S'$, we have  $a \psi^{-1}(S') = b\psi^{-1}(S') $. Then a similar argument establishes that $ A \psi^{-1}(S') = B \psi^{-1}(S') $. Therefore, $ \psi(A) S' = \psi(B) S' $.
\end{proof}

\begin{prop}\label{prop:a22}
 Let $A \in \mathcal{A}_3(S)$. Then $A \in \mathcal{A}_2(S)$ if and only if
 \begin{align}\label{eqn:a23}
 	(\forall B \in P(S))\quad [AS = BS, BA = AB = A] \Longrightarrow B^2 = B.
 \end{align}
\end{prop}
\begin{proof}
 Let $B \in P(S)$ be such that $AS = BS$ and $BA = AB = A$. Then by \eqref{eqn:id}, $(\id A)Y=(\id B)Y$ and $(\id A)(\id B)=\id A$, which gives that $\id B\subseteq (\id B)Y=(\id A)Y$. It follows that for any $\alpha\in \id B$, $\alpha=\beta\gamma$ for some $\beta\in \id A$ and $\gamma \in Y$. So $\alpha=\beta\alpha \in (\id A)(\id B)=\id A$. Thus $\id B\subseteq \id A$. By Lemma \ref{lem:a31}, $ B \subseteq A $.

If $ A \in \mathcal{A}_2(S) $ then by Lemma \ref{lem:a2} we deduce that $ B^2 = B $.
 If $ A \not\in \mathcal{A}_2(S) $, then by Lemma \ref{lem:a3} $ \id A $ has a maximal element $ \omega $ such that $ A_\omega = \{ a, a^0 \} $ is a cyclic group of order two.  Put $ B = A\setminus\{ a^0 \} $. Then $ BA = AB = A $. Also, by Lemma \ref{lem:green2}  $ AS=BS $. But $ B^2 = A\setminus \{ a \} \neq B $.
 Therefore, $ A \in \mathcal{A}_2(S) $ if and only if $ A $ satisfies \eqref{eqn:a23}.
 \end{proof}

\begin{lem}[{\cite[Lemma 2.13]{zgy}}]\label{lem:ss0}
  Let  $ s \in S' $. Then  $ s\psi(S) = s^0\psi(S) $.
\end{lem}

\begin{thm}\label{thm:a24}
  The restriction $ \psi|_{\mathcal{A}_2(S)} $ of $ \psi $ to $ \mathcal{A}_2(S) $ is a bijection from $ \mathcal{A}_2(S) $ onto $ \mathcal{A}_2(S') $.
\end{thm}
\begin{proof}
  Suppose that $ A \in \mathcal{A}_2(S) $. Then $A \in \mathcal{A}_3(S)$ and so, by Theorem \ref{thm:a35}, $ \psi(A) \in \mathcal{A}_3(S') $. If $ \psi(A) \not \in \mathcal{A}_2(S') $, then by Lemma \ref{lem:a3} $ \psi(A) = \bigcup_{\alpha' \in \id\psi(A)}(\psi(A))_{\alpha'}$, where $ \id \psi(A) $ is a subchain of $ Y' $ with a maximal element $ \omega' $, and that $ (\psi (A))_{\omega'} $ is a cyclic group of order two, say $ (\psi(A))_{\omega'} = \{ a, a^0 \} $.
Let $ X = \psi(A) \setminus \{ a^0 \} $. Then $X^2=\psi(A)\setminus\{a\}$ and  $X \psi(A) = \psi(A)X = \psi(A) $. Also, by Lemma \ref{lem:green2}, it is easy to verify that  $X \psi(S) = \psi(A) \psi(S)$.  Hence, $ \psi^{-1}(X)A = A \psi^{-1}(X) = A$ and $\psi^{-1}(X) S = A S$. Since $A \in \mathcal{A}_2(S)$, it follows from Proposition \ref{prop:a22} that $ \psi^{-1}(X^2) = (\psi^{-1}(X))^2 = \psi^{-1}(X) $.  Hence $ X^2 = X $, a contradiction. Therefore, $ \psi(A) \in \mathcal{A}_2(S') $.

A dual argument shows that $ \psi^{-1}(Y) \in \mathcal{A}_2(S)$ if $ Y \in \mathcal{A}_2(S') $.
\end{proof}

Since $ EP(S) $ is the set of idempotents of $ P(S) $,  the natural partial order  $ \leq $ on $ EP(S) $ is defined by
\[
   A \leq B \Longleftrightarrow  A = AB = BA.
\]
We write $ A < B $ whenever $ A \leq B $ and $ A \neq B $. Following Kobayashi \cite{kobayashi},  for any $ A, B \in EP(S) $, we write $ A \twoheadrightarrow B,  A \rightarrow B, A \prec B $ to mean (respectively) that $ A < B $ and  there exists no $ C \in EP(S),  C \in {\mathcal{A}_2}(S), C \in \overline{\mathcal{A}_2}(S)$ (respectively) such that $ A < C < B $.

\begin{rem}
\begin{enumerate}
\item Since $ \overline{\mathcal{A}_2} (S) \subseteq {\mathcal{A}_2}(S)\subseteq EP(S) $, $ A \twoheadrightarrow B \Longrightarrow  A \rightarrow B \Longrightarrow A \prec B$.
\item If $ A \in {\mathcal{A}_2}(S) $ and $ \alpha, \beta \in \id A $ is such that $ \alpha < \beta$, then $ A_{\alpha} < A_{\beta} $.
\end{enumerate}
\end{rem}

\begin{prop}\label{prop:a25}
Let $ A \in \mathcal{A}_2(S) $ and $ B \in EP(S) $ be such that $ A \leq B $. Then
\begin{enumerate}
\item $ B_{\alpha} \subseteq A_{\alpha} $ for all $ \alpha \in \id A  \cap \id B  $. In particular, if $\id B \subseteq \id A$, then $B \subseteq A$.
\item If $\omega \in \id A \cap \id B $ is maximal  both in $\id A$ and in $\id B$, then $ B_\omega = A_\omega$.
\end{enumerate}
\end{prop}
\begin{proof}
(i) Let $ \alpha \in \id A  \cap \id B  $, $a \in A_\alpha$ and $ b \in B_\alpha $. Since $A \leq B$, $b a  \in B_\alpha A_\alpha \subseteq  A_\alpha $. Then $ b a = (ba)^0 $, since $ A \in \mathcal{A}_2(S) $. Hence $ b = (ba)^0 b \in A_\alpha B_\alpha \subseteq A _\alpha$. Therefore, $ B_{\alpha} \subseteq A_{\alpha} $.

(ii) Suppose $\omega\in \id A \cap \id B $ is maximal both in $\id A$ and in $\id B$. Since $A \in \mathcal{A}_2(S)$, from Lemma \ref{lem:a2} we know that $A_\omega$ is either a left or a right zero semigroup and that
\begin{align*}
	A_\omega B_\omega =(AB)_\omega =A_\omega = (BA)_\omega = B_\omega A_\omega.
\end{align*}
By part (i) we have $A_\omega = B_\omega A_\omega B_\omega = B_\omega$.
\end{proof}

By virtue of Proposition \ref{prop:a25}, we can restrict our attention to idempotents of $S$ in the study of order relation on breakable subsemigroups of $S$, and so Proposition \ref{prop:a26}, Theorem \ref{thm:a27} and their proofs, which given by Gan, Zhao and Shao \cite{zhaogs2015} in the context of idempotent semigroups, also valid for  general completely regular semigroups.
\begin{prop}[{\cite[Lemma 2.6]{zhaogs2015}}]\label{prop:a26}
Let $ A\in {\mathcal{A}_2}(S) $ and $ \alpha \in \id A  $. If $\alpha$ is not maximal in $ \id A  $, then $ A \setminus \{a\} \in {\mathcal{A}_2}(S) $ and $ A\twoheadrightarrow A \setminus \{a\}$ for any $ a \in A_{\alpha} $.
\end{prop}

\begin{thm}[{\cite[Proposition 3.2]{zhaogs2015}}]\label{thm:a27}
  The restriction $ \psi|_{\overline{\mathcal{A}_2}(S)} $ of $ \psi $ to $ \overline{\mathcal{A}_2}(S) $ is a bijection from $ \overline{\mathcal{A}_2}(S) $ onto $ \overline{\mathcal{A}_2}(S') $.
\end{thm}

\section{On the structure semilattice of $S$}

In this section, we show that there is a semilattice isomorphism $\theta$ from $Y$ onto $Y'$ such that the restriction $\psi|_{P(S_\alpha)}$ of $\psi$ to $P(S_\alpha)$ is an isomorphism from $P(S_\alpha)$ onto $P(S'_{\theta(\alpha)})$ for all $\alpha \in Y$. .

Gould and Iskra gave the following description of the maximal subgroup $\mathscr{H}$-class of $P(S)$ containing a singleton $\{e\}$, where $e \in E(S)$.
\begin{lem}[{\cite[Lemma 2.1]{gould1984b}}]\label{lem:hclass1}
 Let $ S $ be a semigroup and $ e \in E(S) $. Then $ H_e(P(S)) = H_e(S) $.
\end{lem}

Now we consider the maximal subgroup $\mathscr{H}$-class of $P(S)$ containing a left zero subsemigroup of $S$.
\begin{lem}\label{lem:hclass2}
 Let $ E $ be a left zero subsemigroup of $ S $ and let $ e \in E $. Then $ H_E(P(S)) = \{ Ea \mid a \in H_e(S) \} $.
\end{lem}
\begin{proof}
Clearly, $E=Ee $ is idempotent. Let $a \in H_e(S)$. Then	$E(Ea) = E^2a = Ea$.
 Since $ ae = a $ and $E$ is a left zero semigroup, $aE=aeE=\{ae\}=\{a\}$. Thus $(Ea)E = E(aE) = E a$ and
\[
	(Ea)(Ea^{-1}) = E(aE)a^{-1} = E(aa^{-1}) = Ee = E.
\]
Similarly, $(Ea^{-1})(Ea)=E$.
Thus $Ea \, \mathscr{H} \, E $ in $P(S)$. Therefore,
\[
	\{Ea \mid a \in H_e(S) \} \subseteq H_E(P(S)).
\]

Conversely, if $ A \in H_E(P(S) )$, then   $AE = EA = A$ and there exists $B \in H_E(P(S))$ such that $AB = BA = E$. By Lemma \ref{lem:green1}(ii), $\id A  = \id B  = \id E$. That is, every elements in  $H_E(P(S))$ is contained in the same component $D_e$.

Let $a \in A$. Since $AE=EA=A$, $a=fa_1=a_2g$ for some $a_1, a_2 \in A$ and $f, g\in E$. Thus $f\mathscr{L}g\mathscr{L} a \mathscr{R}f$.  Hence  $a \in H_f(S)$.
On the other hand, for any $f \in E$, $f=ab$ for some $a\in A$ and $b\in B$, since $AB=E$. Thus $f\mathscr{R} a$ and then $f \in H_a(S)$. To summarize, we have shown that
\begin{align}\label{eqn:hclass3}
	(\forall a\in A)(\exists f\in E)a\in H_f(S)\quad \text{and}\quad (\forall f\in E)(\exists a\in A)f\in H_a(S).
\end{align}

Let  $a_1, a_2 \in A \cap H_f(S)$ and $b \in B \cap H_f(S)$. Since $AB = E$, both $a_1 b$ and $ a_2 b $ are  in $ H_f(S) \cap E=\{f\}$, giving that $a_1b = a_2 b =f$. Hence $a_1=a_2$. Consequently, $| A \cap H_f(S)| =1$ for all $f\in E$. In particular,  $| A \cap H_e(S)|=1$, say $A\cap H_e(S)=\{a\}$. Since
\[
	fa \in EA \cap H_f(S)=A \cap H_f(S),
\]
we have $A\cap H_f(S)=\{fa\} $.
This, together with \eqref{eqn:hclass3}, implies that
\[
	A=\bigcup_{f\in E} (A\cap H_f(S))=\bigcup_{f\in E} \{ fa \}=Ea.
\]
Therefore,  $H_E(P(S)) \subseteq \{Ea \mid a \in H_e(S)\}$.
\end{proof}

\begin{thm}\label{thm:theta}
  There exists a semilattice isomorphism $ \theta $ from $ Y $ onto $ Y' $ such that $ \psi|_{P(S_\alpha)} $ is an isomorphism from $ P(S_\alpha) $ onto $ P(S'_{\theta(\alpha)}) $ for all $ \alpha \in Y $.
\end{thm}
\begin{proof}
  Let $ \alpha \in Y $ and $ e \in E(S_\alpha) $. Since $ \{ e \} \in \overline{\mathcal{A}_2}(S) $, it follows from Theorem \ref{thm:a27} that $ \psi (e) \in \overline{\mathcal{A}_2}(S) $. Thus $ \id \psi(e) = \{  \alpha ' \} $ for some $ \alpha' \in Y' $. In the following, we show that $ \id \psi(A) = \{ \alpha' \} $ for all $ A \in P(S_\alpha) $.  Consider the following possibilities of $ A $.

 \emph{Case 1: $ A = \{ a \} $ where $ a \in S_\alpha $.}
 Since  $a \in D_e(S)$, $\{e\}  \mathscr{D}  \{a\} $ in $P(S)$. Thus $ \psi(e)  \mathscr{D}  \psi(a) $ in $ P(S') $. By Lemma \ref{lem:green1}, $ \id\psi(a) = \id \psi(e) = \{\alpha'\} $.

\emph{Case 2: $ A $ is a left zero subsemigroup of $ S_\alpha $}.
Take $a \in A$.  Since $A$ is a left zero semigroup, $ aA = \{a\} $ and $Aa = A$. That is, $A\mathscr{R}\{a\}$ in $P(S)$. Then $ \psi(A) \mathscr{R} \psi(a) $ in $P(S')$.  By Lemma \ref{lem:green1} and case 1, $ \id\psi(A) = \id \psi(a) = \{\alpha'\} $.

\emph{Case 3: $A=H_e(S)$.} Without loss of generality, we may assume that $\psi(e)$ is a left zero semigroup. Take $e'\in \psi(e)$. We show that $\psi(H_e(S))=\psi(e)H_{e'}(S')$ in the following.

%$H_{\psi(e)}(P(S))=\{\psi(e)s \mid s \in H_{e'}(S')\}$.
Let $s \in H_{e'}(S')$. Then by Lemma \ref{lem:hclass2},  $\psi(e)s \mathscr{H} \psi(e)$ in $P(S')$ and so, $e\psi^{-1}(s) \mathscr{H} \{e\}$ in $P(S)$. By Lemma \ref{lem:hclass1}, $e\psi^{-1}(s) \in {H}_e(S)$. Thus, $e\psi^{-1}(s)H_e(S)=H_e(S)$ and then $\psi(e)s\psi(H_e(S))=\psi(H_e(S))$. Therefore, we have
\begin{align}\label{eqn:hclass5}
	\psi(e)H_{e'}(S') \psi(H_e(S))=\bigcup_{s\in H_{e'}(S')}\psi(e)s\psi(H_e(S))=\psi(H_e(S)).
\end{align}

On the other hand, let $a \in H_e(S)$. By Lemma \ref{lem:hclass1} we have $\{a\} \mathscr{H}\{ e\}$ in $P(S)$.  Thus $\psi(a) \mathscr{H} \psi(e)$ in $P(S')$. By  Lemma \ref{lem:hclass2}, $\psi(a) =\psi(e)s$ for some $s \in H_{e'}(S')$. Then
\[
	H_{e'}(S')\psi(a)= H_{e'}(S')\psi(e)s= H_{e'}(S').
\]
Hence, $ \psi^{-1}(H_{e'}(S'))a = \psi^{-1}(H_{e'}(S'))$ and then
\[
	 \psi^{-1}(H_{e'}(S'))H_e(S)=\bigcup_{a\in H_e(S)} \psi^{-1}(H_{e'}(S'))a= \psi^{-1}(H_{e'}(S')).
\]
This implies that $H_{e'}(S')\psi(H_e(S))=H_{e'}(S')$ and then
\begin{align}\label{eqn:hclass6}
	\psi(e)H_{e'}(S')\psi(H_e(S))=\psi(e)H_{e'}(S').
\end{align}

By \eqref{eqn:hclass5} and \eqref{eqn:hclass6}, we have $\psi(H_e(S))=\psi(e)H_{e'}(S')$. Thus
\[
	 \id\psi(H_e(S)) = \id\left( {\psi(e)}H_{e'}(S') \right) = \{ \alpha' \}.
\]

 \emph{Case 4: $ A = H_a (S) $ for some $ a \in S_\alpha $}.
Since $ a \, \mathscr{D} \, e$ in $S$, it is easy to verify $ H_a(S)  \mathscr{D}  H_e(S) $ in $P(S)$, and so $ \psi(H_a(S))  \mathscr{D}  \psi(H_e(S)) $ in $P(S')$.
By Lemma \ref{lem:green1} and case 3,
\[
	\id\psi(H_a(S)) =\id \psi(H_e(S)) = \{\alpha'\}.
\]
%Noting that $ a \, \mathscr{D} \, e$ in $S$, we have that $ a \, \mathscr{R} \, c$ and $c \, \mathscr{L} \, e$ for some $c \in S_\alpha $. That is, there exist $s, s', t, t' \in S^1$ such that
%$	as = c$, $ cs' = a$, $ tc = e$, $ t'e = c$.
%Since the mappings $x\mapsto xs$, $y\mapsto yx'$, $z\mapsto tz$ and $w\mapsto t'w$ are bijections from $H_a(S) \to H_c(S)$, $H_c(S) \to H_a(S)$, $H_c(S) \to H_e(S)$ and $H_e(S) \to H_c(S)$ respectively, it follows  that
%\[
%	H_a(S)s = H_c(S), \quad H_c(S)s' = H_a(S), \quad tH_c(S) = H_e(S), \quad t'H_e(S) = H_c(S).
%\]
%That is, $ H_a(S) \, \mathscr{D} \, H_e(S) $ in $P(S)$, and so $ \psi(H_a(S)) \, \mathscr{D} \, \psi(H_e(S)) $ in $P(S')$.

\emph{Case 5: $ A \subseteq H_a(S) $ for some $ a \in S_\alpha $.}
Clearly, $ A H_a(S) = H_a(S) $ and $ A a^0 = A $. Then $ \psi(A) \psi(H_a(S)) = \psi(H_a(S)) $ and $ \psi(A) \psi(a^0) = \psi(A) $. By cases 1 and 4, $\id\psi(a^0)=\id \psi(H_a(S))=\{\alpha'\}$,  we have
\begin{align*}
  \id\psi(A) = (\id\psi(A))(\id\psi(a^0)) = (\id\psi(A)) (\id\psi(H_a(S))) = \id \psi(H_a(S)) = \{ \alpha' \}.
\end{align*}

\emph{Case 6: $ A = S_\alpha $.}
Since $e S_\alpha e = H_e(S) $ and $S_\alpha e S_\alpha = S_\alpha $, $ \psi(e) \psi(S_\alpha) \psi(e) =  \psi(H_e(S)) $ and $ \psi(S_\alpha) \psi(e) \psi(S_\alpha) = \psi(S_\alpha) $.
Thus,
\[
	\{\alpha'\}=\id\psi(H_e(S))=(\id\psi(e)) (\id\psi(S_\alpha)) (\id\psi(e))= \alpha' (\id\psi(S_\alpha)) \alpha'=(\id\psi(S_\alpha)) \alpha'
\]
 and so
\begin{alignat*}{2}
	\id\psi(S_\alpha) &= (\id\psi(S_\alpha)) (\id\psi(e)) (\id\psi(S_\alpha)) = (\id\psi(S_\alpha)) \alpha' (\id\psi(S_\alpha))  \\
	                                          &= \left( ((\id\psi(S_\alpha)) \alpha'  \right)^2 = \{\alpha'\}^2 = \{\alpha'\}.
\end{alignat*}

\emph{The general case.} Notice  that $ S_\alpha A S_\alpha = S_\alpha $, we have $\psi(S_\alpha) \psi(A) \psi(S_\alpha) = \psi(S_\alpha) $.  Then by case 6,
\[
	\alpha'\left( \id\psi(A) \right) \alpha' =(\id\psi(S_\alpha)) (\id\psi(A)) (\id\psi(S_\alpha))= \id\psi(S_\alpha) = \{\alpha'\}.
\]
This implies that $\alpha' \leq \beta'$ for all $ \beta' \in \id\psi(A) $.
 Now let $ D $ be a left zero subsemigroup of $ S_\alpha $ such that $ \ R_f(S) \cap A \neq \emptyset $ for all $ f \in D$  and  $R_a(S) \cap D \neq \emptyset$ for all $ a \in A$.
By Lemma \ref{lem:green2}, we have $ \psi(D) S' = \psi(A) S' $. It follows by case 2 that
\[
     (\id\psi(A)) Y' = (\id\psi(A)) (\id S') = (\id\psi(D)) (\id S') = \alpha' Y'.
\]
Then $ \id\psi(A) \subseteq (\id\psi(A)) Y' = \alpha' Y' $. From this we deduce that $\beta' \leq \alpha'$ for all $\beta' \in \id\psi(A)$. This, together with the already noted fact that $\alpha' \leq \beta'$ for all $ \beta' \in \id\psi(A) $, implies that $ \id\psi(A) = \{ \alpha' \} $.

We have shown that $ \id\psi(A) = \{ \alpha' \} $ for all $ A \in P(S_\alpha) $. Therefore, the restriction $ \psi|_{P(S_\alpha)}$ of $ \psi $ to $ P(S_\alpha) $ is a mapping from $ P(S_\alpha) $ to $ P(S'_{\alpha'}) $. Dually, by interchanging the roles of $ S $ and $ S' $, and instead of $ \psi $ by $ \psi^{-1}$, we obtain that the restriction $ \psi^{-1}|_{P(S'_{\alpha'})} $ of $ \psi^{-1} $ to $ P(S'_{\alpha'}) $ is a mapping from $ P(S'_{\alpha'}) $ to $ P(S_\alpha) $. Therefore, $ \psi|_{P(S_\alpha)} $ is an isomorphism from $ P(S_\alpha) $ onto $ P(S'_{\alpha'})$.

Define  $ \theta \colon Y \to Y' $ by $\theta(\alpha) = \id\psi(S_\alpha)$. The conclusion of the previous paragraph implies that $ \theta $ is well defined and is bijective. Moreover, $ \theta $ is a homomorphism. In fact, since $S_\alpha S_\beta \subseteq S_{\alpha\beta} $, we have
\begin{align*}
	\theta(\alpha\beta) = \id\psi(S_{\alpha\beta})=\id\psi(S_\alpha S_\beta)= \id(\psi(S_\alpha)\psi(S_\beta)) = (\id\psi(S_\alpha)) (\id\psi(S_\beta))= \theta(\alpha) \theta(\beta).
\end{align*}
Thus, $\theta$ is an isomorphism from semilattice $Y$ onto semilattice $Y'$ such that the restriction $\psi|_{P(S_\alpha)}$ of $\psi$ to $P(S_\alpha)$ is an isomorphism from $P(S_\alpha)$ onto $P(S'_{\theta(\alpha)})$.
\end{proof}
In the remainder of the paper, $ \theta $ always means the semilattice isomorphism from $ Y $ onto $ Y' $ as in Theorem \ref{thm:theta}.

\section{On the singleton sets of $S$}
In this section, we consider  images of  singleton sets of $S$.
By Proposition \ref{prop:cs0} and Theorem \ref{thm:theta}, we conclude that
\begin{prop}\label{prop:cs01}
If $S_\alpha\in \mathcal{CS}_0$, then the restriction of $\psi|_{S_{\alpha}}$ of $\psi$ to $S_\alpha$ is an isomorphism from $S_\alpha$ onto $S'_{\theta(\alpha)}$. Thus, if $D_a \in \mathcal{CS}_0$ for some $a\in S$, then $\psi(a)$ is a singleton.
\end{prop}

In the following of this section, we focus only on the $\mathscr{D}$-classes that are either left or right zero semigroups, we show that even if $S_\alpha$ is a left or a right zero semigroup, we can guarantee that, for some particular cases of $a \in S_\alpha$, $\psi(a)$ is a singleton.

Recall first that there is a natural partial order $\leq$ on  $S$ defined as follows:
\[
	a\leq b \text{ if }  a=eb=bf \text{ for some }  e,f \in E(S),
\]
and it is trivial on a completely simple semigroup (see \cite{petrichreilly}).

\begin{prop}\label{prop:a221}
Let $\{a,b\} \in \mathcal{A}_2(S)$ where $a \in S_\alpha$, $b \in S_\beta$ and $\alpha < \beta$. Then $\psi(\{a,b\})=\psi(a)\cup \psi(b)$ and $\psi(a)$ is a singleton.
\end{prop}
\begin{proof}
By Theorems \ref{thm:a24} and \ref{thm:a27}, $\psi(\{a,b\}) \in \mathcal{A}_2(S)$ with $|\id\psi(\{a,b\})|\geq 2$.
It is obvious that  $\{a\}\leq\{a,b\}\leq\{b\}$. Then $\psi(a)\leq \psi(\{a,b\}) \leq\psi(b)$. From this we deduce that $\theta(\alpha)\leq \alpha' \leq \theta(\beta)$ for all $\alpha' \in \id\psi(\{a,b\})$ and that $\psi(a)\cup\psi(b)$, $\psi(b)\cup \psi(\{a,b\}) \in EP(S')$.

By Proposition \ref{prop:a26}, we have $\{a,b\} \twoheadrightarrow \{b\}$. Hence $ \psi(\{a,b\}) \twoheadrightarrow \psi(b)$.
But it is easy to verify that $\psi(\{a,b\})\leq \psi(b)\cup\psi(\{a,b\})<\psi(b)$.
It follows   that $\psi(\{a,b\})=\psi(b)\cup \psi(\{a,b\})$. Thus $\psi(b)\subseteq \psi(\{a,b\})$ and that $\theta(\beta)\in \id\psi(\{a,b\})$ is maximal in $\id\psi(\{a,b\})$.
Moreover, by Proposition \ref{prop:a25}, we have $(\psi(\{a,b\}))_{\theta(\beta)}=\psi(b)$.

Let $s \in \psi(\{a,b\})\setminus \psi(b)$. By Proposition \ref{prop:a26}, $\psi(\{a,b\})\setminus \{s\} \in \mathcal{A}_2(S')$ and
\[
	\psi(\{a,b\})< \psi(\{a,b\})\setminus\{s\}\leq \psi(b).
\]
Then $\psi(\{a,b\})\setminus\{s\}=\psi(b)$. Therefore $\psi(\{a,b\})\setminus\psi(b)=\{s\}$. Suppose $s \in S'_{\alpha'}$ where $\alpha' <\theta(\beta)$. Clearly,
 $\{s\}\leq \psi(\{a,b\})$ and hence $\psi^{-1}(s)\leq \{a,b\}$. This implies that $\theta^{-1}(\alpha')\leq \alpha$, since
$\{\theta^{-1}(\alpha')\}=\id\psi^{-1}(s)$ and $\id\{a,b\}=\{\alpha,\beta\}$. Thus $\alpha' \leq \theta(\alpha)$. Therefore  $\alpha'=\theta(\alpha)$ and so,
\begin{align}\label{eqn:a222}
	\id\psi(\{a,b\})=\{\theta(\alpha), \theta(\beta)\},\quad (\psi(\{a,b\}))_{\theta(\alpha)}=\{s\}.
\end{align}

In the following, we show that $\psi(\{a,b\})=\psi(a) \cup \psi(b)$.
Since $\psi(a) \cup \psi(b) \in EP(S')$ and $\psi(a) \leq \psi(a) \cup \psi(b) \leq \psi(b)$, we have $\psi^{-1}(\psi(a) \cup \psi(b)) \in EP(S)$
%with $|\id(\psi^{-1}(\psi(a) \cup \psi(b)))|\geq 2$
and
\begin{align}\label{eqn:a223}
	\{a\} \leq \psi^{-1}(\psi(a) \cup \psi(b)) \leq \{b\},
\end{align}
from which we deduce that  $\alpha\leq \gamma \leq \beta$ for all $\gamma\in \id\psi^{-1}(\psi(a) \cup \psi(b))$.

Let $\gamma \in \id \psi^{-1}(\psi(a) \cup \psi(b))$ and $c \in (\psi^{-1}(\psi(a) \cup \psi(b)))_{\gamma}$. Then by \eqref{eqn:a223} we have $ac=ca=a$ and $c=c_1b=bc_2$ for some $c_1, c_2 \in \psi^{-1}(\psi(a) \cup \psi(b))$.  Since $b$ is idempotent, $
c=bc=cb $
 and then $ c^2=(cb)(bc)=cb^2c=cbc=c$.
Hence $\{b, c\} \in \mathcal{A}_2(S)$. It is a routine matter to verify  that $\psi^{-1}(\psi(a) \cup \psi(b)) \leq  \{b,c\}$ and so $\psi(a) \cup \psi(b) \leq \psi(\{b,c\})$. Therefore,
\begin{align}\label{eqn:a224}
	\psi(a) \cup \psi(b)=\psi(\{b,c\})(\psi(a)\cup \psi(b))=\psi(a) \cup \psi(\{b,c\}).
\end{align}
It follows that $\psi(\{b,c\})\subseteq \psi(a)\cup \psi(b)$. By \eqref{eqn:a222}, we have
\[
	\id\psi(\{b,c\})=\{\theta(\beta),\theta(\gamma)\} \subseteq \id(\psi(a) \cup \psi(b))=\{\theta(\alpha),\theta(\beta)\}.
\]
This implies $\theta(\gamma)=\theta(\alpha)$ or $\theta(\gamma) = \theta(\beta)$ and so, $\gamma =\alpha$ or $\gamma = \beta$. Thus
\[
	\id \psi^{-1}(\psi(a)\cup \psi(b))=\{\alpha,\beta\}.
\]

If $c \in (\psi^{-1}(\psi(a) \cup \psi(b)))_\beta$, then by \eqref{eqn:a224} $\psi(\{b,c\})=\psi(b)$ and so $b=c$.  Also,
by Proposition \ref{prop:a25}, we have  $   (\psi^{-1}(\psi(a) \cup \psi(b)))_\alpha=\{a\}$, since $\{a\} \leq \psi^{-1}(\psi(a) \cup \psi(b))$.
Thus $ \psi^{-1}(\psi(a)\cup \psi(b)) =\{a,b\} $ and then $\psi(\{a,b\}) = \psi(a) \cup \psi(b)$. By \eqref{eqn:a222} we know that  $\psi(a)$ is a singleton.
\end{proof}

\begin{prop}\label{prop:nonmax}
If $S_\alpha$ is a left or right zero semigroup and $a \in S_\alpha$ is not a maximal element of $S$ for the natural partial order, then $\psi(a)$ is a singleton.
\end{prop}
\begin{proof}
Since $a$ is not  maximal, there is an idempotent $b \in S_\beta$ such that $a\leq b$ where $\alpha<\beta$. That is, $\{a,b\} \in A_2(S)$.
It follows from Proposition \ref{prop:a221} that $\psi(a)$ is a singleton.
\end{proof}
In the remainder of this section we give some properties of the image of a singleton under $\psi$, which will be of use later.
%To this end we need the following operators $\overline{\psi}$ and $\overline{\psi^{-1}}$ introduced in \cite{YZG}. Define operators $ \overline{\psi} \colon P(S) \to P(S') $ and $ \overline{\psi^{-1}} \colon P(S') \to P(S) $, respectively, by
%\begin{alignat*}{2}
% 	 \overline{\psi} (A) &= \bigcup_{\emptyset \neq B \subseteq A } \psi(B) \qquad&( A \in P(S) ), \\
%  	\overline{\psi^{-1}} (U) &= \bigcup_{\emptyset \neq V \subseteq U} \psi^{-1}(V) \qquad &( U \in P(S') ).
%\end{alignat*}

\begin{lem}[{\cite[Corollary 2.6]{yzg}}] \label{lem:ope}
Let $ \alpha, \beta \in Y $ with $ \alpha > \beta $ and let $ a \in S_\alpha $, $B\in P(S_\beta)$. Then
%$ a B a = A_1 B A_2 $ for every  $ A_1, A_2 \subseteq {\overline{\psi^{-1}}} \,\overline{\psi}(a) $ and $ B \in {P}(S_\beta) $. In particular,
$\psi^{-1}(s)B\psi^{-1}(s)=aBa$ for all $s \in \psi(a)$.
\end{lem}

The following result is proved by Yu, Zhao and Gan \cite{yzg} in the context of idempotent semigroups, however, it continue to be valid for general completely regular semigroups.
\begin{prop}\label{prop:singleton}
Let $ \alpha, \beta \in Y $ with  $ \alpha > \beta $. Then
\begin{enumerate}
\item $ \psi(a) t \psi(a) $ is a singleton for every $ a \in S_\alpha $ and $ t \in S'_{\theta(\beta)}$;
\item $ s \psi(b)s $ is a singleton for every $ s \in S'_{\theta(\alpha)} $ and $ b \in S^{}_\beta$.
\end{enumerate}
\end{prop}
\begin{proof}
(i) Let $a \in S_\alpha$, $t \in S'_{\theta(\beta)}$ and $s \in \psi(a)$. Then $\psi^{-1}(t) \in P(S_\beta)$. By Lemma \ref{lem:ope}, \[ a\psi^{-1}(t)a=\psi^{-1}(s)\psi^{-1}(t)\psi^{-1}(s)=\psi^{-1}(sts).\]
Thus, $\psi(a)t\psi(a)=\{sts\}$ is a singleton.

(ii) Suppose $s \in S'_{\theta(\alpha)}$ and $b \in S_\beta$. Let $a \in \psi^{-1}(s)$.  Applying part (i) to $\psi^{-1}$, we have $\psi^{-1}(s)b\psi^{-1}(s)=\{aba\}$ is a singleton. Hence
$s\psi(b)s =\psi(aba)$.
If $S_\alpha \in \mathcal{CS}_0$, then by Proposition \ref{prop:cs01}, $\psi(aba)$ is a singleton. Otherwise, by Proposition \ref{prop:nonmax}, $\psi(aba)$ is a singleton, since $aba$ is not maximal in $S$.  Therefore, $s\psi(b)s $ is a singleton.
\end{proof}

\section{Global determinism of completely regular semigroups}

We have shown that if $ \psi: P(S) \to P(S') $ is an isomorphism from $ P(S) $ onto $ P(S') $, then there exists a semilattice isomorphism $ \theta: Y \to Y' $ such that the restriction $ \psi|_{P(S_\alpha)} $ of $ \psi $ to $ P(S_\alpha) $ is an isomorphism from $ P(S_\alpha) $ onto $ P(S'_{\theta(\alpha)}) $. Since the class of all completely simple semigroups is globally determined \cite{tamura1986}, we have $S_\alpha \cong S'_{\theta(\alpha)}$.

If $S_\alpha \in \mathcal{CS}_0$, then the restriction $\psi|_{S_\alpha}$ of $\psi$ to $S_\alpha$ is an isomorphism from $S_\alpha$ onto $S'_{\theta(\alpha)}$, otherwise, when $S_\alpha$ is a left or a right zero semigroup, $ \psi | _{S_\alpha}$ need not be so. However, we can show that there is an isomorphism from $S_\alpha$ onto $S'_{\theta(\alpha)}$ which close related to $ \psi $.

Let $S_\alpha$ be a left or a right zero semigroup. First,  we introduce the equivalence relation $\rho_\alpha$ on $S_\alpha$  defined as follows: $(a_1, a_2) \in \rho_\alpha$ if
\begin{enumerate}
\item $a_1$ and $a_2$ are both maximal elements for the natural partial order, or $a_1=a_2$ otherwise;
\item $a_1ba_1=a_2ba_2$ for all $b\in S_\beta$ such that $\alpha>\beta$; and
\item $ca_1c=ca_2c$ for all $c \in S_\gamma$ such that $\alpha<\gamma$.
\end{enumerate}

Since each equivalence relation on a left or a right zero semigroup is a congruence, $\rho_\alpha$ is a congruence on $S_\alpha$.

\begin{prop}\label{prop:rho1}
Let $S_\alpha$ be a left or a right zero semigroup and $a\in S_\alpha$, $ A \subseteq a \rho_\alpha $. Then
\begin{enumerate}
\item $A B A = a B a$ for all $B \in P(S_\beta)$ where $\alpha>\beta$;
\item $CA C = Ca C$ for all $C \in P(S_\gamma)$ where $\alpha <\gamma$.
\end{enumerate}
\end{prop}
\begin{proof}
(i) Let $b \in S_\beta$. Since $A \subseteq a \rho_\alpha $,  $a_1 b a_1 =a_2 b a_2=aba$ for all $a_1, a_2 \in A$. Thus
\[
	a_1ba_2=(a_1 a_2 a_1) b a_2 =a_1 (a_2 a_1 b a_2) = a_1 (a_1 a_1 b a_1) =a_1 b a_1 = a b a,
\]
since $ S_\alpha $ is a left or a right zero semigroup. Hence, for any  $B \in P(S_\beta)$, we have
\[
	 A B A =\{ a_1ba_2 \mid a_1,a_2 \in A, b \in B\}=\{a b a \mid b \in B\} = a B a.
\]

(ii) Let $c_1, c_2 \in S_\gamma$ and $a_1 \in A$. Since $S_\alpha$ is a left or a right zero semigroup and $a_1\rho_\alpha a$,
\[
	c_1a_1c_2=c_1(a_1c_1c_2a_1)c_2=(c_1a_1c_1)(c_2a_2c_2)=c_1ac_1c_2ac_2=c_1ac_2.
\]
Thus for any $C \in P(S_\gamma)$, we have
\[
	CAC=\{c_1a_1c_2 \mid c_1,c_2 \in C, a_1 \in A\}=\{c_1ac_2 \mid c_1,c_2 \in C\} =CaC.
\]
\end{proof}

\begin{prop}\label{prop:rho2}
Let $S_\alpha$ be a left or a right zero semigroup and $a_1, a_2 \in S_\alpha$.
If $a_1 \, \rho_\alpha \, a_2$, then $a_1b=a_2b$ and $ba_1=ba_2$ for all $b \in S_\beta$ where $\alpha>\beta$.
\end{prop}
\begin{proof}
 Suppose  $a_1 \, \rho_\alpha \, a_2$ and $b \in S_\beta$ where $\alpha>\beta$. Then $a_1ba_1=a_2ba_2$ and thus
 \[
 	(a_1b)^2=a_1ba_1b=a_2ba_2b=(a_2b)^2.
 \]
 Hence $a_1b\mathscr{H}a_2b$. Since every $\mathscr{H}$-class of $S$ is a group, it follows that $a_1b=a_2b$. A similar argument shows that $ba_1=ba_2$.
\end{proof}

\begin{prop}\label{prop:rho3}
Let $S_\alpha$ be a left or a right zero semigroup and $a \in S_\alpha $, $s\in \psi(a)$. Then for any $A \in P(S_\alpha)$,  $A \subseteq a \rho_\alpha$ if and only if $\psi(A) \subseteq s\rho_{\theta(\alpha)}$.
\end{prop}
\begin{proof}
Suppose that $A \subseteq a\rho_\alpha$ and $s' \in \psi(A)$.  Let  $t \in S'_{\beta'}$ where $\theta(\alpha)>\beta'$. By Proposition \ref{prop:rho1}(i), $A \psi^{-1}(t) A =a \psi^{-1}(t) a$ and so $\psi(A)t\psi(A)=\psi(a)t\psi(a)$. Thus, by Proposition \ref{prop:singleton}, we have $\psi(A)t\psi(A)=\psi(a)t\psi(a)=\{sts\}$.
Therefore $s'ts'=sts$ for all  $t\in S'_{\beta'}$.

Let $t \in S'_{\gamma'}$ where $\theta(\alpha)< \gamma'$. By Proposition \ref{prop:rho1}(ii), $\psi^{-1}(t)A\psi^{-1}(t) = \psi^{-1}(t)a\psi^{-1}(t)$. Hence,  $t\psi(A)t=t\psi(a)t$. It follows from Proposition \ref{prop:singleton} that $ t\psi(A)t=t\psi(a)t=\{tst\}$, which gives that $ts't=tst$ for all  $t \in S'_{\gamma'}$.

Now we suppose that $a$ is maximal in $S$ for the natural partial order. Then every element of $A$ is maximal. If there exists $s_0\in \psi(A)$ such that $s_0$ is not maximal in $S'$, then $s_0\leq t$ for some $t \in E(S'_{\gamma'})$ where $\theta(\alpha)<\gamma'$. Let $X = \psi(A) \cup \{t\}$. Then $|\id X| \geq 2$ and it is easy to see that  $\psi(A) = X\psi(A) X$. Thus $A =\psi^{-1}(X)A\psi^{-1}(X)$ where $|\id\psi^{-1}(X)|\geq 2$ with $\gamma \geq \alpha$ for all $\gamma \in \id\psi^{-1}(X)$. Take $a \in A$   and $b \in \psi^{-1}(X)\setminus (\psi^{-1}(X))_{\alpha}$. Then $bab \in A$ satisfies $bab \leq b$, in contradiction to the previously noted fact that every element of $A$ is maximal. Hence, every element of $\psi(A)$ is maximal in $S'$.

A dual argument establishes that if every element of $\psi(A)$ is maximal in $S'$, then every element of  $A$ is maximal in $S$.
%Dually, by interchanging the roles of $S$ and $S'$, and instead of $\psi$ by $\psi^{-1}$, we can show that if every element of $\psi(A)$ is maximal in $S'$ for the natural partial order, then every element of  $A$ is maximal in $S$.

Conversely, Suppose $\psi(A)\subseteq s\rho_{\theta(\alpha)}$ and let $a' \in A$, $b \in S_\beta$ where $\alpha>\beta$. Then by Proposition \ref{prop:rho1}, $ \psi(A)\psi(b)\psi(A)=s\psi(b)s $.  Since by Propositions \ref{prop:cs01} and \ref{prop:nonmax} we know that $\psi(aba)$ is a singleton, $s\psi(b)s=\psi(a)\psi(b)\psi(a)=\psi(aba)$. Thus $\psi(AbA)=\psi(aba)$ and so $AbA=aba$. We conclude that $a'ba'=aba$ for all $b \in S_\beta$.

Let $c\in S_{\gamma}$ with $\alpha <\gamma$. Then by Proposition \ref{prop:rho1}, $ \psi(c)\psi(A)\psi(c)=\psi(c)s\psi(c) $. Also, by Proposition \ref{prop:nonmax} we know that $\psi(cac)$ is a singleton. From which we deduce that $\psi(cAc)=\psi(c)s\psi(c)=\psi(cac)$. Hence  $cAc=\{cac\}$. Therefore  $ca'c=cac$ for all $c \in S_\gamma$.

To summarize, we have shown that  $A \subseteq a\rho_\alpha$ if and only if $\psi(A) \subseteq s\rho_{\theta(\alpha)}$.
%
%Conversely, if $\psi(A)\subseteq s\rho_{\theta(\alpha)}$, then by Proposition \ref{prop:rho1}, $ \psi(A)\psi(b)\psi(A)=s\psi(b)s $ and $ \psi(c)\psi(A)\psi(c)=\psi(c)s\psi(c) $ for all $b \in S_{\beta}$ and $c\in S_{\gamma}$ where $\beta < \alpha <\gamma$. Notice that if $S_\beta\in \mathcal{CS}_0$, then by Proposition \ref{prop:cs01} $\psi(bab)=\psi(b)\psi(a)\psi(b)$ is a singleton, otherwise $bab$ is not maximal in $S$, then by Proposition \ref{prop:nonmax}  $\psi(b)\psi(a)\psi(b)$ is a singleton. It follows that $\psi(AbA)=s\psi(b)s=\psi(aba)$ in each case. Also, by Proposition \ref{prop:nonmax} we know that $\psi(cac)=\psi(c)\psi(a)\psi(c)$ is a singleton, since $cac$ is not maximal in $S$. Thus $\psi(cAc)=\psi(c)s\psi(c)=\psi(cac)$.
%Hence $AbA=\{aba\}$ and $cAc=\{cac\}$. So we have  $a'ba'=aba$ and $ca'c=cac$ for all $a' \in A$. Therefore, $A \subseteq a\rho_\alpha$.
\end{proof}

From Proposition \ref{prop:rho3}, we know that the restriction $\psi|_{P(a\rho_\alpha)}$ of $\psi$ to $P(a\rho_\alpha)$ is a bijection from $P(a\rho_\alpha)$ onto $P(s\rho_{\theta(\alpha)})$, where $a \in S_\alpha$, $s \in \psi(a)$. Thus $a\rho_\alpha$ and $s\rho_{\theta(\alpha)}$ have the same cardinality . Hence $a\rho_\alpha  \cong s\rho_{\theta(\alpha)}$ and so, there is an isomorphism $\xi_{\bar{a}}$ from $a\rho_\alpha$ onto $s\rho_{\theta(\alpha)}$, where $\bar{a}$ denotes the $\rho$-class $a\rho_\alpha$.  Moreover, from the definition of $\rho_\alpha$, we know that if  $ a $ is not maximal in $S$, then $\xi_{\bar{a}} =\psi(a)$. Now we  define  $\eta_\alpha \colon S_\alpha \to S'_{\theta(\alpha)}$ by
\begin{align}\label{eq:etaalpha}
	\eta_\alpha = \bigcup_{a \in S_\alpha} \xi_{\bar{a}}.
\end{align}
Then $\eta_\alpha$ is an isomorphism from $S_\alpha$ onto $S'_{\theta(\alpha)}$ such that $\eta_\alpha(a) \in s\rho_{\theta(\alpha)}$ for all $a \in S_\alpha$ where $s \in \psi(a)$. In particular, if $a$ is not maximal in $S$, then $\eta_\alpha(a) = \psi(a)$

\begin{thm}
 The class $ \mathcal{CR} $ of completely regular semigroups is globally determined.
\end{thm}
\begin{proof}
 Assume that both $ S = [Y; S_\alpha ] $ and $ S' = [Y'; S'_{\alpha'}] $ are completely regular semigroups, and $ \psi: P(S) \to P(S') $ is an isomorphism. By Theorem \ref{thm:theta}, there is a semilattice isomorphism $ \theta: Y \to Y' $ such that for every $ \alpha \in Y $, the restriction $ \psi|_{P(S_\alpha)} $ of $ \psi $ to $ S_\alpha $ is an isomorphism from $ P(S_\alpha) $ onto $ P(S'_{\theta(\alpha)}) $. Since the class of completely simple semigroups is globally determined, we have $ S_\alpha \cong S'_{\theta(\alpha)} $. Thus if $|Y|=1$ then $|Y'|=1$ and $ S \cong S' $. In the following, we consider the case where $ | Y |, | Y' | \geq 2 $.

 For each $ \alpha \in Y $, define $ \eta_\alpha \colon S_\alpha \to S'_{\theta(\alpha)} $ as follows: if $S_\alpha$ is either a left or a right zero semigroup, then we define $\eta_\alpha$ as \eqref{eq:etaalpha}; otherwise, we define $ \eta_\alpha $ is equal to the restriction $\psi|_{S_\alpha}$ of $\psi$ to $S_\alpha$. Let $\eta = \bigcup_{\alpha \in Y} \eta_\alpha$. Then $\eta$ is a bijection from $S$ onto $S'$.
Now let $ a \in S_\alpha $ and $ b \in S_\beta $. We show that $ \eta(ab) =\eta(a) \eta(b) $. Clearly, it holds in the case of $ \alpha = \beta $.

Consider the following cases of $ \alpha $ and $ \beta $.

{ \it Case 1: $ \alpha < \beta $. } (i) If $S_\alpha \in \mathcal{CS}_0$,  then $\eta(a) = \eta_\alpha (a) = \psi(a)$ and
\[
	 \eta(ab) = \eta_{\alpha}(ab) = \psi(ab) = \psi(a) \psi(b).
\]
If  $S_\beta \in \mathcal{CS}_0$, then $\eta(b)=\eta_\beta(b) = \psi(b)$ and so $ \eta(ab) = \psi(a) \psi(b) = \eta(a) \eta(b)$. If $S_\beta$ is either a left or a right zero semigroup, then $\eta(b) = \eta_\beta(b) \in s\rho_{\theta(\beta)}$ where  $s \in \psi(b)$. By Proposition \ref{prop:rho2}, we have
\[
	\eta(a)\eta(b)=\eta(a)s=\psi(a)\psi(b)=\eta(ab).
\]

(ii) If $S_\alpha$ is a left zero semigroup, then $ab=a$ and $\eta(a)\eta(b)=\eta(a)$. Hence
\[
	\eta(ab) = \eta(a)=\eta(a)\eta(b).
\]

(iii) If $S_\alpha$ is a right zero semigroup, then $ab=bab<b$ and $\eta(a)=\eta_\alpha(a) \in s\rho_{\theta(\alpha)}$ where $s \in \psi(a)$. Hence, by the definition of $\eta_\alpha$,
\[
	\eta(ab)=\eta_\alpha(ab)=\psi(ab)=\psi(a)\psi(b).
\]
If $S_\beta \in \mathcal{CS}_0$, then $\eta(b)=\eta_\beta(b)=\psi(b)$. Hence,
\begin{align*}
	&\eta(a)\eta(b)=\eta(b)\eta(a)\eta(b)& &(\text{since } S'_{\theta(\alpha)} \text{ is a right zero semigroup})\\
	=&\eta(b)s\eta(b)=\eta(b)\psi(a)\eta(b)&&(\text{since } \eta(a) \in s\rho_{\theta(\alpha)} )\\
	=&\psi(b)\psi(a)\psi(b)=\psi(a)\psi(b)&&(\text{since } S'_{\theta(\alpha)} \text{ is a right zero semigroup})\\
	=&\eta(ab).&&
\end{align*}

If $S_\beta$ is either a left or  right zero semigroup, then $\eta(b)=\eta_\beta(b) \in t\rho_{\theta(\beta)}$ where $t \in \psi(b)$. Therefore,
\begin{align*}
	&\eta(a)\eta(b)=\eta(b)\eta(a)\eta(b)& &(\text{since } S'_{\theta(\alpha)} \text{ is a right zero semigroup})\\
	=&t\eta(a)t=\psi(b)\psi(a)\psi(b)&&(\text{since } \eta(b) \in t\rho_{\theta(\alpha)} )\\
	=&\psi(a)\psi(b)=\eta(ab).&&(\text{since } S'_{\theta(\alpha)} \text{ is a right zero semigroup})
\end{align*}
%
%\[
%	\eta(a)\eta(b)=\eta(b)\eta(a)\eta(b)=t\eta(a)t=\psi(b)\psi(a)\psi(b)=\psi(a)\psi(b)=\eta(ab).
%\]

{\it Case 2: $ \alpha > \beta $.}  The dual of Case 1.

{\it Case 3: $ \alpha $ and $ \beta $ are incomparable. }
If $S_{\alpha\beta} \in\mathcal{CS}_0$, then $ \eta(ab) = \psi(ab) $. Otherwise, $ ab $ is not a maximal element of $S$, thus $ \eta(ab) = \eta_{\alpha\beta} (ab) = \psi(ab) $. Then, a similar argument of case 1 shows that  $\eta(ab)=\eta(a)\eta(b)$ in this case.
\end{proof}

\end{document}